\documentclass[10pt]{amsart}
\usepackage{amssymb,MnSymbol}
\usepackage{amsthm,amsmath}
\usepackage{cite}

\title{Acz\'elian $n$-ary semigroups}\thanks{Communicated by Mikhail Volkov.}

\author{Miguel Couceiro}
\address{Mathematics Research Unit, FSTC, University of Luxembourg \\
6, rue Coudenhove-Kalergi, L-1359 Luxembourg, Luxembourg}
\email{miguel.couceiro[at]uni.lu }

\author{Jean-Luc Marichal}
\address{Mathematics Research Unit, FSTC, University of Luxembourg \\
6, rue Coudenhove-Kalergi, L-1359 Luxembourg, Luxembourg}
\email{jean-luc.marichal[at]uni.lu }

\date{October 18, 2011}

\begin{document}

\theoremstyle{plain}
\newtheorem{theorem}{Theorem}[section]
\newtheorem{lemma}[theorem]{Lemma}
\newtheorem{proposition}[theorem]{Proposition}
\newtheorem{corollary}[theorem]{Corollary}
\newtheorem{fact}[theorem]{Fact}
\newtheorem*{main}{Main Theorem}

\theoremstyle{definition}
\newtheorem{definition}[theorem]{Definition}
\newtheorem{example}[theorem]{Example}

\theoremstyle{remark}
\newtheorem*{conjecture}{Conjecture}
\newtheorem{remark}{Remark}
\newtheorem*{note}{Note}
\newtheorem{claim}{Claim}

\newcommand{\N}{\mathbb{N}}                     
\newcommand{\R}{\mathbb{R}}                     
\newcommand{\Q}{\mathbb{Q}}
\newcommand{\B}{\mathbb{B}}                     
\newcommand{\C}{\mathbb{C}}
\newcommand{\I}{\mathbb{I}}
\newcommand{\J}{\mathbb{J}}
\newcommand{\bfs}{\mathbf{s}}
\newcommand{\bfx}{\mathbf{x}}
\newcommand{\bfy}{\mathbf{y}}
\newcommand{\bfz}{\mathbf{z}}
\newcommand{\med}{\mathrm{med}}

\begin{abstract}
We show that the real continuous, symmetric, and cancellative $n$-ary semigroups are topologically order-isomorphic to additive real $n$-ary
semigroups. The binary case ($n=2$) was originally proved by Acz\'el \cite{Acz49}; there symmetry was redundant.
\end{abstract}



\maketitle

\section{Introduction}

Let $I$ be a nontrivial real interval (i.e., nonempty and not a singleton) and let $n\geqslant 2$ be an integer. Recall that an $n$-ary function
$f\colon I^n\to I$ is said to be \emph{associative} if it solves the following system of $n-1$ functional equations:
\begin{eqnarray*}
\lefteqn{f\big(x_1,\ldots,f(x_i,\ldots,x_{i+n-1}),x_{i+n},\ldots,x_{2n-1}\big)}\\
&=& f\big(x_1,\ldots,x_i,f(x_{i+1},\ldots,x_{i+n}),\ldots,x_{2n-1}\big),\quad i=1,\ldots,n-1.
\end{eqnarray*}
The pair $(I,f)$ is then called an $n$-ary semigroup (see D\"ornte~\cite{Dor28} and Post~\cite{Pos40}).

A function $f\colon I^n\to I$ is said to be \emph{cancellative} if it is one-to-one in each variable; that is, for every $k\in
[n]=\{1,\ldots,n\}$ and every $\bfx=(x_1,\ldots,x_n)\in I^n$ and $\bfx'=(x'_1,\ldots,x'_n)\in I^n$,
$$
\mbox{$\big( x_i=x'_i~~$ for all $i\in [n]\setminus\{k\}~~$ and $~~f(\bfx)=f(\bfx')\big)\quad\Rightarrow\quad x_k=x'_k\,$.}
$$
Also, a function $f\colon I^n\to I$ is said to be \emph{symmetric} if, for every permutation $\sigma$ on $[n]$, we have $f(x_1,\ldots,x_n)=f(x_{\sigma(1)},\ldots,x_{\sigma(n)})$.

In this paper we present a complete description of those associative functions $f\colon I^n\to I$ which are continuous, symmetric, and
cancellative. Our main result can be stated as follows.

\begin{main}
A function $f\colon I^n\to I$ is continuous, symmetric, cancellative, and associative if and only if there exists a continuous and strictly
monotonic function $\varphi\colon I\to J$ such that
\begin{equation}\label{eq:s7af5}
f(\bfx)=\varphi^{-1}\left(\sum_{i=1}^n\varphi(x_i)\right),
\end{equation}
where $J$ is a real interval of one of the forms $]{-\infty},b[$, $]{-\infty},b]$, $]a,\infty[$, $[a,\infty[$ or $\R ={]{-\infty},\infty[}$
$(b\leqslant 0\leqslant a)$. For such a function $f$, $I$ is necessarily open at least on one end. Moreover, $\varphi$ can be chosen to be strictly increasing. In other
words, the $n$-ary semigroup $(I,f)$ is topologically order-isomorphic to the $n$-ary semigroup $(J,+)$.
\end{main}

The binary case ($n=2$) of the Main Theorem, for which symmetry is not needed, was first stated and proved by J.\ Acz\'el~\cite{Acz49}. A
shorter, more technical proof of Acz\'el's result was then provided by Craigen and P\'ales~\cite{CraPal89} (see also \cite{Acz04} for a recent
survey). The corresponding binary semigroups are called \emph{Acz\'elian} (see Ling~\cite[Section 3.2]{Lin65}).

We say that an $n$-ary semigroup is \emph{Acz\'elian} if it satisfies the assumptions of the Main Theorem. Thus the Main Theorem provides an
explicit description of the class of Acz\'elian $n$-ary semigroups. Although this result is not a trivial derivation of the binary case, we
prove it by following more or less the same steps as in \cite{CraPal89}.

The following example shows that the symmetry assumption is necessary for every odd integer $n\geqslant 3$.

\begin{example}\label{ex:asd6}
Let $n\geqslant 3$ be an odd integer. The function $f\colon \R^n\to \R$, defined by
$$
f(\bfx)=\sum_{i=1}^n (-1)^{i-1}x_i\, ,
$$
is continuous, cancellative, and associative. However, it cannot be of the form (\ref{eq:s7af5}) with a continuous and strictly monotonic
function $\varphi$. Indeed, if the latter would be the case, then by identifying the variables, we would have $f(x^n)=x$ and hence
$\varphi(x)=\varphi(f(x^n))=n\,\varphi(x)$, a contradiction.
\end{example}

This paper is organized as follows. In Section 2 we show how $n$-ary associative functions can be extended to associative functions of certain
higher arities. In Section 3 we provide the proof of the Main Theorem.

To avoid cumbersome notation, we henceforth regard tuples $\bfx$ in $I^n$ as $n$-strings over $I$ and we write $|\bfx|=n$. The $0$-string or
\emph{empty} string is denoted by $\varepsilon$ so that $I^0=\{\varepsilon\}$. We denote by $I^*$ the set of all strings over $I$, that is,
$I^*=\bigcup_{n\in \N}I^n$, where $\N=\{0,1,2,\ldots\}$. Moreover, we consider $I^*$ endowed with concatenation for which we adopt the
juxtaposition notation. For instance, if $\bfx\in I^n$, $y\in I$, and $\bfz\in I^m$, then $\bfx y\bfz\in I^{n+1+m}$.

\begin{remark}
Using this notation, we immediately see that a function $f\colon I^n\to I$ is associative if and only if we have $f(\bfx\,
f(\bfy)\bfz)=f(\bfx'\, f(\bfy')\bfz')$ for every $\bfx\bfy\bfz,\bfx'\bfy'\bfz'\in I^{2n-1}$ such that $\bfy,\bfy'\in I^n$ and $\bfx\bfy\bfz
=\bfx'\bfy'\bfz'$. Similarly, $f$ is cancellative if and only if, for every $\bfx\bfz\in I^{n-1}$ and every $y,y'\in I$, the equality $f(\bfx
y\bfz)=f(\bfx y'\bfz)$ implies $y=y'$.
\end{remark}

For $x \in I$, we also use the short-hand notation $x^m=x\cdots x\in I^m$. Given a function
$g\colon I^*\to I$, we denote by $g_m$ the restriction of $g$ to $I^m$, i.e.\ $g_m:=g|_{I^m}$. We convey that $g_0$ is defined by $g_0(\varepsilon)=\varepsilon$.

\section{Associative extensions}

Recall that a binary function $f\colon I^2\to I$ is said to be \emph{associative} if
$$
f(f(xy)z)=f(xf(yz))\qquad\mbox{for all $x,y,z\in I$}.
$$
Using an infix notation, we can also write this property as
$$
(x\diamond y)\diamond z=x\diamond (y\diamond z)\qquad\mbox{for all $x,y,z\in I$}.
$$
Since associativity expresses that the order in which variables are bracketed is not relevant, it can be easily extended to functions $g\colon
I^*\to I$ by defining
$$
g_m(x_1\cdots x_m)=x_1\diamond\cdots\diamond x_m
$$
for every integer $m\geqslant 2$. The latter definition can be reformulated in prefix notation as $g_2=f$ and
\begin{equation}\label{eq:saf76sf5}
g_m(x_1\cdots x_m)=g_2(g_2(\cdots g_2(g_2(g_2(x_1x_2)x_3)x_4)\cdots) x_m)
\end{equation}
for every $m>2$. Equivalently, we may write $g_2=f$ and $$g_m(x_1\cdots x_m)=g_2(g_{m-1}(x_1\cdots x_{m-1})x_m)$$ for every $m>2$.

Note that the unary function $g_1$ is not involved in this construction and so it could be chosen arbitrarily. However, as we will see in
Proposition~\ref{prop:as8df65}, it is convenient to ask $g_1$ to satisfy the following condition:
\begin{equation}\label{eq:saf765}
\mbox{$g_1\circ g=g$ \quad and\quad $g(\bfx\, g_1(y)\bfz)=g(\bfx y\bfz)$ $~~$ for all $\bfx y\bfz\in I^*$.}
\end{equation}

\begin{definition}
A function $g\colon I^*\to I$ is said to be \emph{associative} if
\begin{enumerate}
\item[$(i)$] $g_2$ is associative,

\item[$(ii)$] condition (\ref{eq:saf76sf5}) holds for every $m>2$ and every $x_1,\ldots,x_m\in I$, and

\item[$(iii)$] condition (\ref{eq:saf765}) holds.
\end{enumerate}
\end{definition}

By definition, an associative function $g\colon I^*\to I$ can always be constructed from a binary associative function $f\colon I^2\to I$ by
defining $g_2=f$, using (\ref{eq:saf76sf5}), and choosing a unary function $g_1$ satisfying (\ref{eq:saf765}) (e.g., the identity
function).\footnote{Note that $g_1$ necessarily solves the idempotency equation $g_1\circ g_1=g_1$.} Such a function $g$, which is completely
determined by $g_1$ and $g_2=f$, will be called an \emph{associative extension} of $f$.

The following proposition provides concise reformulations of associativity of functions $g\colon I^*\to I$ and justifies condition
(\ref{eq:saf765}). We will prove a more general statement in Proposition~\ref{prop:as8df65p}. The equivalence of assertions $(ii)$--$(iv)$ was
proved in \cite{CouMar11}.

\begin{proposition}\label{prop:as8df65}
Let $g\colon I^*\to I$ be a function. The following assertions are equivalent.
\begin{enumerate}
\item[$(i)$] $g$ is associative.

\item[$(ii)$] $g(\bfx\, g(\bfy)\bfz)=g(\bfx' g(\bfy')\bfz')$ for every $\bfx\bfy\bfz,\bfx'\bfy'\bfz'\in I^*$ such that
$\bfx\bfy\bfz=\bfx'\bfy'\bfz'$.

\item[$(iii)$] $g(\bfx\, g(\bfy)\bfz)=g(\bfx\bfy\bfz)$ for every $\bfx\bfy\bfz\in I^*$.

\item[$(iv)$] $g(g(\bfx)g(\bfy))=g(\bfx\bfy)$ for every $\bfx\bfy\in I^*$.
\end{enumerate}
\end{proposition}

For any integer $n\geqslant 2$, define the sets
$$
A_n=\{m\in\N: m\equiv 1~(\mathrm{mod}~n-1)\}\qquad\mbox{and}\qquad I^{(n)}=\bigcup_{m\in A_n}I^m=I\times(I^{n-1})^*.
$$
Just as associativity for binary functions can be extended to functions $g\colon I^*\to I$, one can also extend the associativity of $n$-ary
functions to functions $g\colon I^{(n)}\to I$ as follows.\footnote{This construction is inspired from D\"ornte~\cite{Dor28} and
Post~\cite{Pos40}.} Given an associative function $f\colon I^n\to I$, we define $g\colon I^{(n)}\to I$ as $g_n=f$ and
\begin{equation}\label{eq:saf76sf5p}
g_m(x_1\cdots x_m)=g_n(g_n(\cdots g_n(g_n(x_1\cdots x_n)x_{n+1}\cdots x_{2n-1})\cdots)x_{m-n+2}\cdots x_m)
\end{equation}
for every $m\in A_n$ and $m>n$. Equivalently, we may write $g_n=f$ and $$g_m(x_1\cdots x_m)=g_n(g_{m-n+1}(x_1\cdots x_{m-n+1})x_{m-n}\cdots x_m)$$
for every $m\in A_n$ and $m>n$.

Once again, the unary function $g_1$ can be chosen arbitrarily. However, we ask $g_1$ to satisfy the following condition:
\begin{equation}\label{eq:saf765p}
\mbox{$g_1\circ g=g$ \quad and\quad $g(\bfx\, g_1(y)\bfz)=g(\bfx y\bfz)$ $~~$ for all $\bfx y\bfz\in I^{(n)}$.}
\end{equation}

\begin{definition}\label{de:a7ds4f}
A function $g\colon I^{(n)}\to I$ is said to be \emph{$n$-associative} if
\begin{enumerate}
\item[$(i)$] $g_n$ is associative,

\item[$(ii)$] condition (\ref{eq:saf76sf5p}) holds for every $m\in A_n$, $m>n$, and every $x_1,\ldots,x_m\in I$, and

\item[$(iii)$] condition (\ref{eq:saf765p}) holds.
\end{enumerate}
\end{definition}

By definition, an $n$-associative function $g\colon I^{(n)}\to I$ can always be constructed from an $n$-ary associative function $f\colon I^n\to
I$ by defining $g_n=f$, using (\ref{eq:saf76sf5p}), and choosing a unary function $g_1$ satisfying (\ref{eq:saf765p}) (e.g., the identity
function). Such a function $g$, which is completely determined by $g_1$ and $g_n=f$, will be called an \emph{$n$-associative extension} of $f$.

\begin{example}
From the ternary associative function $f\colon\R^3\to\R$, defined by $f(x_1x_2x_3)=x_1-x_2+x_3$, we can construct the $3$-associative extension
$g\colon \R^{(3)}\to \R$ as
$$
g_m(x_1\cdots x_m)=\sum_{i=1}^m(-1)^{i-1}x_i\qquad (m\geqslant 3,~\mbox{odd}),
$$
for which (\ref{eq:saf765p}) provides the unique solution $g_1=\mathrm{id}$.
\end{example}

The following proposition generalizes Proposition~\ref{prop:as8df65} and provides concise reformulations of $n$-associativity of functions
$g\colon I^{(n)}\to I$ and justifies condition (\ref{eq:saf765p}).

\begin{proposition}\label{prop:as8df65p}
Let $g\colon I^{(n)}\to I$ be a function. The following assertions are equivalent.
\begin{enumerate}
\item[$(i)$] $g$ is $n$-associative.

\item[$(ii)$] $g_1\circ g=g$ and $g(\bfx\, g(\bfy)\bfz)=g(\bfx' g(\bfy')\bfz')$ for every $\bfx\bfy\bfz,\bfx'\bfy'\bfz'\in I^{(n)}$ such
that $\bfy,\bfy'\in I^{(n)}$ and $\bfx\bfy\bfz=\bfx'\bfy'\bfz'$.

\item[$(iii)$] $g(\bfx\, g(\bfy)\bfz)=g(\bfx\bfy\bfz)$ for every $\bfx\bfy\bfz\in I^{(n)}$ such that $\bfy\in I^{(n)}$.

\item[$(iv)$] $g_1\circ g=g$ and $g(g(\bfx_1)\cdots g(\bfx_n))=g(\bfx_1\cdots\bfx_n)$ for every $\bfx_1,\ldots,\bfx_n\in I^{(n)}$.
\end{enumerate}
\end{proposition}

\begin{proof}
Implications $(iii)\Rightarrow (i)$, $(iii)\Rightarrow (ii)$, and $(iii)\Rightarrow (iv)$ are easy to verify.

To prove $(ii)\Rightarrow (iii)$ simply take $\bfy'=\bfx\bfy\bfz$ (i.e., $\bfx'\bfz'=\varepsilon$).

Let us now prove that $(iv)\Rightarrow (iii)$. Let $\bfx\bfy\bfz\in I^{(n)}$ such that $\bfy\in I^{(n)}$. We write $\bfx\, g(\bfy)\bfz=t_1\cdots
t_m$, with $t_k=g(\bfy)$. By $(iv)$ we have
$$
g(\bfx\, g(\bfy)\bfz) = g(t_1\cdots t_m) = g\big(g(t_1)\cdots g(t_{n-1})g(t_n\cdots t_m)\big).
$$
If $k\leqslant n-1$, then
\begin{eqnarray*}
g(\bfx\, g(\bfy)\bfz) &=& g\big(g(t_1)\cdots g(t_k)\cdots g(t_{n-1})g(t_n\cdots t_m)\big)\\
&=& g\big(g(t_1)\cdots g(\bfy)\cdots g(t_{n-1})g(t_n\cdots t_m)\big)%
~=~ g(\bfx\bfy\bfz).
\end{eqnarray*}
If $k\geqslant n$, we proceed similarly with $g(t_n\cdots t_m)$, unless $n=m$ in which case the result follows immediately.

Let us establish that $(i)\Rightarrow (iii)$. We only need to prove that $g(\bfx\, g(\bfy)\bfz)=g(\bfx\bfy\bfz)$ for every $\bfx\bfy\bfz\in
I^{(n)}$ such that $|\bfy|\geqslant 2$ and $|\bfx\bfz|\geqslant 1$. Using (\ref{eq:saf76sf5p}) twice and the associativity of $g_n$, we can
rewrite the function $\bfx\bfy\bfz\mapsto g(\bfx\, g(\bfy)\bfz)$ in terms of nested $g_n$'s only. Then, using the associativity of $g_n$ again,
we can move all the $g_n$'s to the left to obtain the right-hand side of $(\ref{eq:saf76sf5p})$, which reduces to $g(\bfx\bfy\bfz)$.

To illustrate, consider the following example with $n=3$:
\begin{eqnarray*}
g(x_1x_2x_3\, g(x_4x_5x_6x_7x_8)x_9) &=& g(x_1\, g(x_2x_3\, g(x_4\, g(x_5x_6x_7)x_8))x_9)\\
&=& g(g(g(g(x_1x_2x_3)x_4x_5)x_6x_7)x_8x_9)\\
&=& g(x_1x_2x_3x_4x_5x_6x_7x_8x_9).\qedhere
\end{eqnarray*}
\end{proof}

\begin{remark}
Proposition~\ref{prop:as8df65} follows from Proposition~\ref{prop:as8df65p}. Note that the condition $g_1\circ g=g$ is not needed in assertions
$(ii)$ and $(iv)$ of Proposition~\ref{prop:as8df65} since $I^*$ is used instead of $I^{(n)}$, thus allowing the use of the empty string
$\varepsilon$.
\end{remark}

\section{Proof of the Main Theorem}

It is easy to show that the condition in the Main Theorem is sufficient. To show that the condition is necessary, let $I$ be a nontrivial real
interval, let $f\colon I^n\to I$ be a continuous, symmetric, cancellative, and associative function, and let $g\colon I^{(n)}\to I$ be the
unique $n$-associative extension of $f$ such that $g_1=\mathrm{id}$ (see the observation following Definition~\ref{de:a7ds4f}).

\begin{claim}\label{claim:sdf3458ds}
$f$ is strictly increasing in each variable.
\end{claim}

\begin{proof}
Since $f$ is continuous and cancellative, it must be strictly monotonic in each variable. Suppose it is strictly decreasing in the first
variable. Then, by associativity, for every $\bfy\in I^{n-1}$, $u\in I$, and $\mathbf{v}\in I^{n-2}$, the unary function $x\mapsto
f(f(x\bfy)u\mathbf{v})=f(x\, f(\bfy u)\mathbf{v})$ is both strictly increasing and strictly decreasing, which leads to a contradiction. Thus $f$
must be strictly increasing in the first variable and hence in every variable by symmetry.
\end{proof}

An element $e\in I$ is said to be an \emph{idempotent} for $f$ if $f(e^n)=e$. For instance, any real number is an idempotent for the function
defined in Example~\ref{ex:asd6}.

\begin{claim}\label{claim:sdf8ds}
There cannot be two distinct idempotents for $f$.
\end{claim}

\begin{proof}
Otherwise, if $d$ and $e$ were distinct idempotents, we would have
$$
f(d\, e^{n-1})=f(f(d^n)\, e^{n-1})=f(d\, f(d^{n-1}e)\, e^{n-2})
$$
and hence (by cancellation), $e=f(d^{n-1}e)=f(e\, d^{n-1})$. Similarly, $d=f(e^{n-1}d)=f(d\, e^{n-1})$. Now, if $e<d$, then $d=f(d\, e^{n-1})<
f(d^{n-1}e)=e$ (by Claim~\ref{claim:sdf3458ds}), a contradiction. We arrive at a similar contradiction if $d<e$.
\end{proof}

Because of Claim~\ref{claim:sdf8ds}, there is a $c\in I$ such that either $c<f(c^n)$ or $c>f(c^n)$. We assume w.l.o.g.\ that the former holds and fix such a $c$. The latter case can be dealt with similarly.

\begin{claim}\label{lem:2}
For all fixed $x\in I$, we have $x<f(x\, c^{n-1})$. Thus the sequence $x_m=f(x_{m-1}c^{n-1})$ strictly increases, and $\lim x_m\notin I$ (hence
$\lim x_m=\sup I$ and $I$ is open from above).
\end{claim}

\begin{proof}
Since $c<f(c^n)$, we have $f(c\, x^{n-1})<f(f(c^n)\, x^{n-1})=f(c\, f(c^{n-1}x)\, x^{n-2})$ and hence (by strict monotonicity)
$x<f(c^{n-1}x)=f(x\, c^{n-1})$. Thus $x_m=f(x_{m-1}c^{n-1})>x_{m-1}$. If $\lim x_m=x'$ and $x'\in I$, continuity gives the following:
$$
x'=\lim x_m=\lim f(x_{m-1}c^{n-1})=f(\lim x_{m-1}c^{n-1})=f(x'c^{n-1}),
$$
a contradiction. Thus $x'\notin I$, so $\lim x_m=\sup I$.
\end{proof}

Hereinafter we work on the extended real line so that suprema of arbitrary sets exist and all monotone sequences converge.

%

\begin{claim}\label{claim:9asf7a}
Let $x\in I$ and let $j,k,p,q\in\N$ such that $j+1,k,p,q+1\in A_n$. Then we have
$$
g(c^p)>g(x\, c^q)\quad\Leftrightarrow\quad g(c^{kp})>g(x^kc^{kq})\quad\Leftrightarrow\quad g(c^{p+j})>g(x\, c^{q+j}).
$$
The same equivalence holds if ``$<$'' or ``='' replaces ``$>$''.
\end{claim}

\begin{proof}
Assume that $g(c^p)>g(x\, c^q)$. Then, by Proposition~\ref{prop:as8df65p}$(iv)$, Claim~\ref{claim:sdf3458ds}, and symmetry, we have $g(c^{kp})=g(g(c^p)^k)>g(g(x\, c^q)^k)=g(x^kc^{kq})$,
which proves the first equivalence (since the same conclusion clearly holds if ``$<$'' or ``='' replaces ``$>$''). For the second equivalence,
assume again that $g(c^p)>g(x\, c^q)$. Then, as before, we have $g(c^{p+j})=g(g(c^p)\, c^j)>g(g(x\, c^q)\, c^j)=g(x\, c^{q+j})$.
\end{proof}

Let $x$ be any fixed element of $I$. Define $S_x$ to be the subset of all rational numbers $r$ for which there exist $k,p,q\in\N$ such that
$k,p,q+1\in A_n$, $g(c^p)>g(x^kc^q)$, and $r=(p-q)/k$. Now, if $r=(p-q)/k=(p'-q')/k'$, then we have $pk'+q'k=p'k+qk'$ and it follows from
Claim~\ref{claim:9asf7a} that
\begin{eqnarray*}
g(c^p) > g(x^kc^q) & \Leftrightarrow & g(c^{pk'}) > g(x^{kk'}c^{qk'})\\
& \Leftrightarrow & g(c^{pk'+q'k}) > g(x^{kk'}c^{qk'+q'k})\\
& \Leftrightarrow & g(c^{p'k+qk'}) > g(x^{kk'}c^{q'k+qk'})\\
& \Leftrightarrow & g(c^{p'k}) > g(x^{kk'}c^{q'k})\\
& \Leftrightarrow & g(c^{p'}) > g(x^{k'}c^{q'}).
\end{eqnarray*}
Hence $S_x$ is in fact the subset of rational numbers $r$ for which every representation $r=(p-q)/k$ with $k,p,q+1\in A_n$ satisfies $g(c^p) >
g(x^kc^q)$.

\begin{claim}\label{claim:e6w7r}
The set $S=\{\frac{p-q}k : k,p,q+1\in A_n\}$ is dense in $\R$.
\end{claim}

\begin{proof}
For every $a,b\in\N$, the sequence
$$
x_m=\frac{1\pm a\, m\,(n-1)}{1+b\, m\, (n-1)}
$$
of $S$ converges to $\pm a/b$. Thus $S$ is dense in $\Q$ and hence (by transitivity) in $\R$.
\end{proof}

\begin{claim}\label{claim:7sd65}
Any two numbers $r,r'\in S$ may be written $r=(p-q)/k$, $r'=(p'-q)/k$ for suitable $k,p,p',q+1\in A_n$.
\end{claim}

\begin{proof}
Let $r=(p-q)/k$ and $r'=(p'-q')/k'$, with $k,k',p,p',q+1,q'+1\in A_n$. Assume w.l.o.g.\ that $r'>r$. Setting $\tilde{k}=k\, k'$,
$\tilde{q}=|\tilde{k}\, r-1|$, $\tilde{p}=\tilde{k}\, r+\tilde{q}$, and $\tilde{p}'=\tilde{k}\, r'+\tilde{q}$, we have
$r=(\tilde{p}-\tilde{q})/\tilde{k}$, $r'=(\tilde{p}'-\tilde{q})/\tilde{k}$ with $\tilde{k},\tilde{p},\tilde{p}',\tilde{q}+1\in A_n$.
\end{proof}

\begin{claim}\label{claim:1}
$S_x$ is a nonempty, proper, and upper subset of $S$ (``upper'' means that if $r\in S_x$ and $r'\in S$, $r'>r$, then $r'\in S_x$).
\end{claim}

\begin{proof}
To show that $S_x$ is an upper subset, let $r=(p-q)/k\in S_x$ and $r'=(p'-q)/k>r$ (cf.\ Claim~\ref{claim:7sd65}). Then $p'>p$ and, since
$p,p'\in A_n$, we have $p'=p+j(n-1)$ for some integer $j\geqslant 1$. Using the definition of $S_x$ and the first part of Claim~\ref{lem:2}, we
obtain
\begin{eqnarray*}
g(x^kc^q) ~<~ g(c^p) & < & g(g(c^p)\, c^{n-1}) ~=~ g(c^pc^{n-1})\\
& < & g(g(c^pc^{n-1})\, c^{n-1}) ~=~ g(c^pc^{2(n-1)})\\
& < & \cdots\\
& < & g(c^pc^{j(n-1)}) ~=~ g(c^{p'}).
\end{eqnarray*}
Hence $r'\in S_x$. Now, by Claim~\ref{lem:2}, $\lim f(c^{m(n-1)+1})=\sup I>g(x\, c^{n-1})$, and hence there is some $p\in A_n$ with
$g(c^p)>g(x\, c^{n-1})$. Hence $r=(p-(n-1))/1\in S_x$, and so $S_x$ is nonempty. Similarly, since $\lim g(x\, c^{m(n-1)})=\sup I$, there must a
$q$ such that $q+1\in A_n$ and $g(c)<g(x\, c^q)$, and so $(1-q)/1\notin S_x$.
\end{proof}

Now, by Claim \ref{claim:1}, $S_x$ is precisely the set of elements in $S$ which are greater than (and possibly equal to) $\inf S_x$. Using this
fact, let $\varphi\colon I\to\R$ be the function given by
\[
\varphi(x):= \inf S_x\, .
\]

\begin{claim}\label{claim:2}
If $g(c^p)=g(x^kc^q)$, then $\varphi(x)=(p-q)/k$. In particular, $\varphi(c)=1$.
\end{claim}

\begin{proof}
Note that $g(c^p)=g(x^kc^q)$ implies $r=(p-q)/k\notin S_x$. Moreover, by Claim~\ref{claim:1} it follows that if $r'=(p'-q)/k>r$ (resp.\ $r'<r$),
then $g(c^{p'})>g(c^p)=g(x^kc^q)$ (resp.\ $g(c^{p'})<g(c^p)=g(x^kc^q)$), and hence $r'\in S_x$ (resp.\ $r'\not\in S_x$). Thus $\inf S_x=(p-q)/k$
by Claim~\ref{claim:e6w7r}. For the last claim just note that $g(c^{q+1})=g(c\, c^q)$.
\end{proof}

\begin{claim}\label{claim:4}
We have $\varphi(g(x_1\cdots x_n))=\sum_{i=1}^n\varphi(x_i)$ for every $x_1,\ldots, x_n\in I$.
\end{claim}

\begin{proof}
Let $r_i=(p_i-q)/k >\varphi(x_i)$ for all $i\in [n]$. Then $g(c^{p_i})>g(x_i^kc^q)$, and by Proposition~\ref{prop:as8df65p}$(iv)$, Claim~\ref{claim:sdf3458ds}, and symmetry, we have
\[
g(c^{\sum_{i=1}^n{p_i}})=g(g(c^{p_1})\cdots g(c^{p_n}))>g(g(x_1^kc^q)\cdots g(x_n^kc^q))= g(g(x_1\cdots x_n)^kc^{nq}).
\]
By Claim~\ref{claim:2}, $(\sum_{i=1}^n{p_i}-nq)/k\in S_{g(x_1\cdots x_n)}$. Thus $\sum_{i=1}^n r_i>\varphi(g(x_1\cdots x_n))$. Similarly, if
$r_i\leqslant\varphi(x_i)$ for all $i\in [n]$, then $\sum_{i=1}^n r_i\leqslant\varphi(g(x_1\cdots x_n))$. The result then follows from
Claim~\ref{claim:e6w7r}.
\end{proof}

\begin{claim}
$\varphi$ is nondecreasing.
\end{claim}

\begin{proof}
Suppose $y>x$ and $(p-q)/k\in S_y$. Then $g(c^p)>g(y^kc^q)>g(x^kc^q)$ and hence $S_y\subseteq S_x$ and so $\varphi(y)=\inf S_y \geqslant\inf
S_x=\varphi(x)$.
\end{proof}

\begin{claim}\label{claim:3}
$\varphi$ is continuous.
\end{claim}

\begin{proof}
Since $\varphi$ is nondecreasing, the only possible sort of discontinuity is a gap discontinuity. Hence, if $\varphi$ is discontinuous, there must
exist $x,y\in I$, say $x<y$, and an interval, and thus a rational $r\notin\varphi(I)$, such that $\varphi(x)<r<\varphi(y)$. Now if $r=(p-q)/k$,
then $g(x^kc^q)<g(c^p)\leqslant g(y^kc^q)$. By continuity of $g_{k+q}$, there is $t\in {]x,y]}$ such that $g(c^p)=g(t^kc^q)$. By Claim~\ref{claim:2}
it then follows that $\varphi(t)=r$, which yields the desired contradiction.
\end{proof}

\begin{claim}\label{claim:5}
$\varphi$ is strictly increasing.
\end{claim}

\begin{proof}
For the sake of contradiction, suppose that there are $x,y\in I$ such that $x<y$ and $\varphi(x)=\varphi(y)=a$. Since $\varphi$ is
nondecreasing, there is an interval $I'$ containing $x$ and $y$, and such that $\varphi(z)=a$, for all $z\in I'$. Let $I'$ be the largest interval
having this property, and set $t=\sup I'$. If $t\notin I$, then for every $z>x$, $\varphi(z)=a$. Now $g(x\, c^{n-1})>x$ (by Claim~\ref{lem:2})
and hence $a=\varphi(g(x\, c^{n-1}))=a+(n-1)>a$ (by Claim~\ref{claim:4}), a contradiction. Thus $t\in I$, and $\varphi(t)=a$ by
Claim~\ref{claim:3}. We have $g(x\, t^{n-1})<g(t^{n})$ and, by Claim~\ref{lem:2}, there exists $q$ such that $q+1\in A_n$ and $g(t^{n})<g(x\,
c^{q(n-1)})=g(x\, g(c^q)^{n-1})$ and $g(c^q)>t$. By continuity of $g_n$, there is $z\in I$ such that $t<z<g(c^q)$ (and so $z\notin I'$) and
$g(x\, z^{n-1})=g(t^{n})$. Thus
\[
a+(n-1)\,\varphi(z)=\varphi(x)+(n-1)\,\varphi(z)=\varphi(g(x\, z^{n-1}))=\varphi(g(t^{n}))=n\, \varphi(t)=n\, a\, ,
\]
and we obtain $\varphi(z)=a$, so $z\in I'$, a contradiction.
\end{proof}

Thus $\varphi$ is a continuous strictly increasing $n$-ary semigroup homomorphism and, by Claim~\ref{claim:4}, its range $J$ is a connected real additive $n$-ary semigroup. Hence the only possibilities for $J$ are $]{-\infty},b[$,
$]{-\infty},b]$, $]a,\infty[$, $[a,\infty[$ or ${]{-\infty},\infty[}$ $(b\leqslant 0\leqslant a)$; see final comments in \cite{CraPal89}. This completes the proof of the Main Theorem.\qed

\begin{remark}
The function $\varphi$ is determined up to a multiplicative constant, that is, with $\varphi$ all functions $\psi = r\,\varphi$ ($r\neq 0$) belong
to the same function $f$, and only these; see the ``Uniqueness'' section in \cite{Acz04}.
\end{remark}

\begin{remark}
An $n$-ary semigroup $(I,f)$ is said to be \emph{reducible to} (or \emph{derived from}) a binary semigroup $(I,\diamond)$ if there is an
associative extension $g\colon I^*\to I$ of $\diamond$ such that $g_n=f$; that is, $f(x_1\cdots x_n)=x_1\diamond\cdots\diamond x_n$ (see
\cite{Dor28,Pos40}). Dudek and Mukhin \cite{DudMuk06} showed that an $n$-ary semigroup is reducible if and only if we can adjoint an $n$-ary neutral element to it. This shows that the $n$-ary semigroup given in
Example~\ref{ex:asd6} is not reducible since we cannot adjoint any $n$-ary neutral element (for an alternative proof, see \cite{MarMat11}). However, the Main Theorem shows that every Acz\'elian $n$-ary semigroup is reducible and hence we can always adjoint an $n$-ary neutral element to it (if $0\in J$, then the neutral element is $e=\varphi^{-1}(0)$; otherwise fix $e\notin I$ and extend $\varphi$ to $\varphi'\colon I\cup\{e\}\to J\cup\{0\}$ by the rule $\varphi'(x)=\varphi(x)$ if $x\in I$ and $\varphi'(e)=0$).
\end{remark}

\section*{Acknowledgments}

The authors wish to thank Judita Dasc\u{a}l, Pierre Mathonet, and Michel Rigo for helpful comments and suggestions. We would also like to thank the referee for bringing to our attention reference \cite{DudMuk06}, which provides a necessary and sufficient condition for an $n$-ary semigroup to be reducible. This research is supported by the internal
research project F1R-MTH-PUL-09MRDO of the University of Luxembourg.

\end{document}